\documentclass[12pt]{amsart}

\usepackage{amsmath,amssymb,amsfonts,amscd,latexsym}
\usepackage[all]{xy}
\usepackage{graphics,epsfig,colortbl}
\input xy
\xyoption{all}

\usepackage{vmargin} 
\setmarginsrb{20mm}{10mm}{20mm}{16mm}{20mm}{7mm}{20mm}{8mm}

\usepackage[pdftex,bookmarks,colorlinks,breaklinks]{hyperref} 

\newtheorem{theorem}{Theorem}[section]

\newtheorem{proposition}[theorem]{Proposition}
\newtheorem{corollary}[theorem]{Corollary}
\newtheorem{lemma}[theorem]{Lemma}

\newtheorem{question}[theorem]{Question}

\theoremstyle{definition}
\newtheorem{definition}[theorem]{Definition}

\theoremstyle{remark}

\newcommand{\Z}{\mathbb{Z}}

\newcommand{\id}{\operatorname{id}}

\title[Countable compactness and $A$-equivalence]
{Countable Compactness and Free Abelian Topological Groups}
\author[Z. Huang]{Zhouxiang Huang}
 \address[Z. Huang]
{Institute of Mathematics, Nanjing Normal University, Nanjing 210024, China}
\email{2715383574@qq.com}
\date{September 8, 2026}
\begin{document}
\begin{abstract}
Let $A(X)$ denote the free Abelian topological group over a Tychonoff
space $X$. We prove that if $A(X)$ and $A(Y)$ are topologically
isomorphic and $X$ is countably compact, then $Y$ is also countably compact.
Thus countable compactness is an $A$-invariant and, consequently, an
$M$-invariant; this answers Open Problem~7.10.1 of Arhangel'skii and
Tkachenko. 
\end{abstract}

\maketitle

\section{Introduction}

Throughout the paper all spaces are Tychonoff and all topological groups
are Hausdorff. For a space $X$, let $F(X)$ and $A(X)$ denote the Markov
free topological group and the free Abelian topological group over $X$,
respectively. Spaces $X$ and $Y$ are called $M$-equivalent
(respectively, $A$-equivalent) if $F(X)$ and $F(Y)$ (respectively,
$A(X)$ and $A(Y)$) are topologically isomorphic.

By \cite[Proposition~7.10.1]{AT}, $M$-equivalence implies
$A$-equivalence. Compactness and pseudocompactness are $A$-invariants;
see \cite[Theorems~7.10.10 and 7.10.11]{AT}. The intermediate property
led Arhangel'skii and Tkachenko to the following question
\cite[Open Problem~7.10.1]{AT}.

\begin{question}\label{q:AT}
Suppose that $X$ and $Y$ are $M$-equivalent spaces and that $X$ is
countably compact. Must $Y$ be countably compact?
\end{question}

The question was repeated in \cite[Question~2]{RHO}. A principal partial
result obtained there says that if all finite powers of $X$ are countably
compact and $Y$ is $A$-equivalent to $X$, then all finite powers of $Y$
are countably compact \cite[Theorem~1.15]{RHO}. Since countable
compactness is not preserved by finite products, this result does not
settle Question~\ref{q:AT}.

The purpose of this paper is to prove the following stronger assertion.

\begin{theorem}\label{thm:intro-main}
Let $X$ and $Y$ be Tychonoff spaces. If $A(X)$ and $A(Y)$ are
topologically isomorphic, then
\[
 X\text{ is countably compact}
 \quad\Longleftrightarrow\quad
 Y\text{ is countably compact}.
\]
Consequently, countable compactness is both an $A$-invariant and an
$M$-invariant.
\end{theorem}
\section{Notation and preliminary facts}\label{sec:preliminaries}

If $S$ is a subset of an Abelian group, $\langle{S}\rangle$ denotes the algebraic
subgroup generated by $S$. Every element $g\in A(X)$ has a unique reduced
representation
\[
 g=n_1x_1+\cdots+n_rx_r,
\]
where $x_1,\ldots,x_r$ are distinct points of $X$ and
$n_1,\ldots,n_r$ are non-zero integers. Put
\[
 \ell_X(g)=\sum_{i=1}^r |n_i|
 \quad\text{and}\quad
 A_n(X)=\{g\in A(X):\ell_X(g)\le n\}.
\]
The neutral element has length zero.

We shall use the following standard facts.

\begin{proposition}\label{prop:known}
Let $S$ and $T$ be Tychonoff spaces.
\begin{enumerate}
\item If $F(S)\cong F(T)$, then $A(S)\cong A(T)$.
\item If $A(S)\cong A(T)$, then $S$ is pseudocompact if and only if
      $T$ is pseudocompact.
\item Every pseudocompact subspace $P$ of $A(S)$ is contained in
      $A_n(S)$ for some $n<\omega$.
\item If $S$ is pseudocompact, the canonical homomorphism
      $A(S)\to A(\beta S)$ is a topological monomorphism.
\end{enumerate}
\end{proposition}

\begin{proof}
Assertion~(1) is Proposition~7.10.1 of \cite{AT}, and
assertion~(2) is Theorem~7.10.10 therein. Assertions~(3) and~(4) are
Corollaries~7.5.4 and~7.7.5 of the same monograph. The Abelian form of
assertion~(3) also appears as Proposition~1.4 of \cite{RHO}.
\end{proof}

If $K$ is compact, then $A_n(K)$ is compact. Indeed, $A_n(K)$ is the
continuous image, under addition, of the $n$th power of the topological
sum $K\oplus(-K)\oplus\{0\}$.

\begin{definition}\label{def:basis}
A subspace $K$ of a Hausdorff Abelian topological group $G$ is called a
\emph{topological free Abelian basis} of $G$ if the homomorphism
$A(K)\to G$ induced by the inclusion $K\hookrightarrow G$ is a
topological isomorphism.
\end{definition}

In particular, a topological free Abelian basis is an algebraic free
Abelian basis. This algebraic part of the definition will be used when
coefficients are compared.

\section{A positive answer to Question~\ref{q:AT}}
We use $\omega$ as the countable discrete index set and put
$\omega^*=\beta\omega\setminus\omega$. For $D\subseteq\omega$, write
$D^*=\overline D^{\beta\omega}\cap\omega^*$.

We recall the Wallman construction from \cite[Section~3.6]{Eng}.
For a $T_1$-space $X$, let $\mathcal D(X)$ be the lattice of closed
subsets of $X$. A \emph{closed ultrafilter} is a maximal proper filter
in $\mathcal D(X)$, and $wX$ denotes the set of all closed ultrafilters.
Identify $x\in X$ with the principal closed ultrafilter
\[
 p_x=\{F\in\mathcal D(X):x\in F\}.
\]
For open $U\subseteq X$, put
\[
 \mathcal W(U)=\{p\in wX:F\subseteq U\text{ for some }F\in p\}.
\]
These sets form a base for the Wallman topology. The resulting space
$wX$ is quasi-compact and $T_1$, and the above identification embeds
$X$ as a dense subspace \cite[Theorem~3.6.21]{Eng}.
Moreover, $wX$ is Hausdorff if and only if $X$ is normal
\cite[Theorem~3.6.22]{Eng}; in that case it is canonically equivalent
to $\beta X$ \cite[Corollary~3.6.23]{Eng}.
We use this identification for normal spaces below.
For the discrete space $\omega$, closed ultrafilters are ordinary
ultrafilters, and
\[
 \overline D^{\beta\omega}=\{p\in\beta\omega:D\in p\}
 \qquad(D\subseteq\omega).
\]
These sets form a clopen base of $\beta\omega$; consequently, the sets
$D^*$ with $D$ infinite form a clopen base of $\omega^*$.

\begin{lemma}\label{1}
Let $X$ be a normal space and let $A$ be a closed subset of $X$.
Under the identification $wX=\beta X$, a closed ultrafilter $p$
satisfies $p\in\overline A^{\beta X}$ if and only if $A\in p$.
\end{lemma}
\begin{proof}
Suppose that $A\in p$. Every basic neighbourhood $\mathcal W(U)$
of $p$ contains a member $F\in p$ in the sense that $F\subseteq U$.
Since $A\cap F\ne\emptyset$, this neighbourhood meets $A$. Thus
$p\in\overline A^{\beta X}$.
Conversely, if $A\notin p$, maximality gives $F\in p$ with
$A\cap F=\emptyset$. Then $\mathcal W(X\setminus A)$ is a
neighbourhood of $p$ disjoint from $A$, so
$p\notin\overline A^{\beta X}$.
\end{proof}

\begin{lemma}\label{2}
Let $X$ be a normal space, let $Y$ be a Tychonoff space, and let
$s:X\to Y$ be continuous. Let $\overline s:\beta X\to\beta Y$
be its continuous extension. For every closed ultrafilter $p$ on $X$,
\[
 \{\overline s(p)\}
 =\bigcap_{A\in p}\overline{s(A)}^{\beta Y}.
\]
\end{lemma}
\begin{proof}
For each $A\in p$, Lemma~\ref{1} gives
$p\in\overline A^{\beta X}$, so continuity implies
$\overline s(p)\in\overline{s(A)}^{\beta Y}$.
If $y\in\beta Y\setminus\{\overline s(p)\}$, regularity of
$\beta Y$ yields an open neighbourhood $U$ of $\overline s(p)$
with $y\notin\overline U^{\beta Y}$. Put
\[
 A=s^{-1}(\overline U^{\beta Y}\cap Y).
\]
This set is closed in $X$ and contains
$X\cap\overline s^{-1}(U)$. Since $X$ is dense in $\beta X$ and
$\overline s^{-1}(U)$ is a neighbourhood of $p$, we have
$p\in\overline A^{\beta X}$. Thus $A\in p$ by Lemma~\ref{1},
whereas
$\overline{s(A)}^{\beta Y}\subseteq\overline U^{\beta Y}$
does not contain $y$. This proves the equality.
\end{proof}

\begin{lemma}\label{3}
Let $X$ be a countably compact subspace of a compact Hausdorff space
$K$, and let $s:\omega\to X$ be a map. Denote by
$\overline s:\beta\omega\to K$ the continuous extension of $s$,
viewed as a map into $K$. Then
\[
 \Gamma(s,X)=\{p\in\omega^*:\overline s(p)\in X\}
\]
is dense in $\omega^*$.
\end{lemma}
\begin{proof}
Fix an infinite set $D\subseteq\omega$. We shall find a point of
$D^*\cap\Gamma(s,X)$.
If some $x\in X$ has $D\cap s^{-1}(x)$ infinite, choose this $x$.
Otherwise $s(D)$ is infinite, and countable compactness gives an
accumulation point $x\in X$ of $s(D)$. In either case,
\[
 D\cap s^{-1}(V)\text{ is infinite}
\]
for every open neighbourhood $V$ of $x$ in $X$.
The family of all these sets, together with all cofinite subsets of
$\omega$, has the finite intersection property. Extend it to an
ultrafilter $p$ on $\omega$. Then $p$ is free and $D\in p$.
For every $A\in p$ and every neighbourhood $V$ of $x$ in $X$, the
intersection $A\cap D\cap s^{-1}(V)$ is nonempty. Hence
$x\in\overline{s(A)}^K$ for every $A\in p$. Applying
Lemma~\ref{2} to the map $s:\omega\to K$ and using $\beta K=K$,
we obtain
\[
 \{\overline s(p)\}=\bigcap_{A\in p}\overline{s(A)}^K=\{x\}.
\]
Thus $p\in D^*\cap\Gamma(s,X)$. Since the sets $D^*$ with $D$
infinite form a base of $\omega^*$, the assertion follows.
\end{proof}
\begin{lemma}\label{4}
Let $X$ be a nonempty Baire space, and let $E_j$ be a closed subset
of $X$ for each $j\in J:=\{1,\ldots,m\}$. For every $I\subseteq J$,
including $I=\varnothing$, put
\[
 P_I=\left(\bigcap_{i\in I}E_i\right)
     \cap\left(\bigcap_{j\in J\setminus I}(X\setminus E_j)\right).
\]
An intersection indexed by the empty set is understood to be $X$.
In particular,
\[
 P_{\varnothing}=\bigcap_{j\in J}(X\setminus E_j)
                 =X\setminus\bigcup_{j\in J}E_j,
 \qquad
 P_J=\bigcap_{j\in J}E_j.
\]
Then there exists a subset $I\subseteq J$ such that
$P_I$ contains a non-empty open subset of $X$.
\end{lemma}
\begin{proof}
If $P_{\varnothing}\ne\varnothing$, then $P_{\varnothing}$ is itself
open in $X$, since $J$ is finite and each $E_j$ is closed. Thus the
conclusion holds with $I=\varnothing$. We may therefore assume that
$P_{\varnothing}=\varnothing$.
Since each $P_I$ is locally closed.
It suffices to prove each locally closed subset with empty interior is no where dense.
Indeed, supposed $P=F\cap V$ has empty interior but $P$ is not no where dense. 
Thus, there exists a non-empty open subset $W\subseteq \overline {P}$. 
Since $\overline {P}\subseteq F$ and $\emptyset\neq W\cap V\subseteq V$.
Thus, $W\cap V\subseteq P$, a contradiction.

Since $X=\bigcup_{I\subseteq J}P_I$ and $X$ is Baire space.
Hence, there exists a  $I\subseteq J$ such that $P_I$ has non-empty interior.
\end{proof}
\begin{definition}
Let $X$ be a Tychonoff space, $(b_n)_{n\in\omega}$ a sequence in $X$,
and $p$ a free ultrafilter on $\omega$. A point $b\in\beta X$ is called
a \emph{$p$-limit} of $(b_n)$ in $\beta X$, written
\[
  b=p\text{-}\lim_{n\in\omega}b_n,
\]
if
\[
  \{n\in\omega:b_n\in U\}\in p
\]
for every neighbourhood $U$ of $b$ in $\beta X$.
\end{definition}
\begin{lemma}\label{5}
Let $X$ be a Tychonoff space, $(b_n)_{n\in\omega}$ a sequence in $X$,
and $p$ a free ultrafilter on $\omega$. Then there exists a unique
point $b(p)\in\beta X$ such that
\[
  b(p)=p\text{-}\lim_{n\in\omega} b_n
\]
in $\beta X$.
\end{lemma}

\begin{proof}
We give the compactness argument with $C=\beta X$; it applies to a
sequence in any compact Hausdorff space $C$. For each $A\in p$, put
\[
F_A=\overline{\{b_n:n\in A\}}^{\,C}.
\]
Since finite intersections of members of $p$ are nonempty, the family
$\{F_A:A\in p\}$ has the finite intersection property. Compactness
therefore yields a point $b\in\bigcap_{A\in p}F_A$.

For any open neighbourhood $U$ of $b$, let
$S_U=\{n:b_n\in U\}$. If $S_U\notin p$, then
$A=\omega\setminus S_U\in p$, whence
\[
b\in F_A\subseteq C\setminus U,
\]
a contradiction. Thus $S_U\in p$ for every such $U$, proving that
$b=p\text{-}\lim_n b_n$. Uniqueness follows from the Hausdorff
property: two distinct limits would have disjoint neighbourhoods
whose corresponding index sets both belong to $p$, which is impossible.

Different ultrafilters may yield different limits. If $b_n\in B\subseteq C$,
the argument only guarantees $b(p)\in C$; finding a free ultrafilter
$p$ with $b(p)\in B$ requires an additional argument.
\end{proof}
\begin{lemma}\label{lem:ultrafilter-limit-continuity}
Let $C$ be a compact Hausdorff space and let $(y_n)_{n\in\omega}$
be a sequence in $C$. For $p\in\omega^*$, put
\[
 y(p)=p\text{-}\lim_{n\in\omega}y_n.
\]
Then $y:\omega^*\to C$ is continuous. More precisely, if
$s:\omega\to C$ is given by $s(n)=y_n$ and
$\overline s:\beta\omega\to C$ is its continuous extension, then
$y=\overline s|_{\omega^*}$.
Consequently, for any two sequences $(u_n)$ and $(v_n)$ in $C$, the set
\[
 \{p\in\omega^*:p\text{-}\lim_n u_n=p\text{-}\lim_n v_n\}
\]
is closed in $\omega^*$.
\end{lemma}
\begin{proof}
Since $\omega$ is discrete, $s$ is continuous; the universal property
of $\beta\omega$ gives the continuous extension $\overline s$.
Fix $p\in\omega^*$ and an open neighbourhood $U$ of $\overline s(p)$
in $C$. Choose $D\subseteq\omega$ such that
\[
 p\in\overline D^{\beta\omega}\subseteq\overline s^{-1}(U).
\]
Then $D\in p$ and $D\subseteq\{n:y_n\in U\}$, so
$\{n:y_n\in U\}\in p$. Thus $\overline s(p)$ is the $p$-limit
of $(y_n)$. Uniqueness of the limit, as in Lemma~\ref{5}, gives
$y(p)=\overline s(p)$, proving continuity.
For the final assertion, the two limit maps $u,v:\omega^*\to C$
are continuous, and their equalizer is
$(u,v)^{-1}(\Delta_C)$, where
$\Delta_C=\{(z,z):z\in C\}$ is closed in $C\times C$ because
$C$ is Hausdorff.
\end{proof}

\begin{theorem}\label{main}
Let $G$ be a Hausdorff free abelian topological group, and let $K$ and $C$ be
compact topological free Abelian bases of $G$. Let $X\subseteq K$, and let $B\subseteq C$.
Supposed that $\langle{X}\rangle=\langle{B}\rangle$ as algebraic subgroups of $G$. 
If $X$ is countably compact, then $B$ is countably compact.
\end{theorem}
\begin{proof}
Let $D$ be a countably infinite subset of $B$, and choose a one-to-one enumeration $D=\{b_n:n\in\omega\}$. 
We shall find an accumulation point of $D$ in $B$.

Since $C$ is a compact subspace of $A(K)$, Proposition~\ref{prop:known}(3) gives $C\subseteq A_N(K)$ for some $N<\omega$. 
After passing to a subsequence and renumbering, we may write
\[
 b_n=\sum_{i=1}^{r}q_i x_{n,i},
\tag{2.2}\label{eq:b-K}
\]
where $r$ and the non-zero integers $q_1,\ldots,q_r$ do not depend on $n$, and $x_{n,1},\ldots,x_{n,r}$ are distinct points of $K$ for each $n$. 
Since $b_n\in\langle{B}\rangle=\langle{X}\rangle$. 
The uniqueness of the reduced representation relative to the basis $K$ gives
\[
 x_{n,i}\in X
 \qquad(n\in\omega,\ 1\le i\le r).
\tag{2.3}\label{eq:x-X}
\]

\smallskip
The compact set $K\subseteq A(C)$ is contained in $A_M(C)$ for some
$M<\omega$. Since $r$ is finite, a further common subsequence permits us
to fix, for every $i\le r$, an integer $m_i$ and non-zero integers
$s_{i,1},\ldots,s_{i,m_i}$ such that
\[
 x_{n,i}=\sum_{j=1}^{m_i}s_{i,j}c_{n,i,j},
\tag{2.4}\label{eq:x-C}
\]
where the points $c_{n,i,1},\ldots,c_{n,i,m_i}$ are distinct for fixed
$n$ and $i$. Since $x_{n,i}\in X\subseteq\langle{B}\rangle$, uniqueness of the
reduced representation relative to $C$ gives
\[
 c_{n,i,j}\in B.
\tag{2.5}\label{eq:c-B}
\]

\smallskip
Apply Lemma~\ref{lem:ultrafilter-limit-continuity} to $(b_n)$ and
each $(c_{n,i,j})$ in $C$, and to each $(x_{n,i})$ in $K$.
For $p\in\omega^*$, denote the resulting limits by
\[
 b(p)=p\text{-}\lim_n b_n,\qquad
 x_i(p)=p\text{-}\lim_n x_{n,i},\qquad
 c_{i,j}(p)=p\text{-}\lim_n c_{n,i,j}.
\tag{2.6}\label{eq:limits}
\]
Thus the maps $b,c_{i,j}:\omega^*\to C$ and
$x_i:\omega^*\to K$ are continuous and are the restrictions of the
corresponding Stone--\v{C}ech extensions. Continuity of addition
implies
\[
 b(p)=\sum_{i=1}^{r}q_i x_i(p)
\tag{2.7}\label{eq:limit-one}
\]
and
\[
 x_i(p)=\sum_{j=1}^{m_i}s_{i,j}c_{i,j}(p).
\tag{2.8}\label{eq:limit-two}
\]
At the limit, different letters may coalesce and opposite coefficients
may cancel. The next calculation records both effects.

\smallskip
For each fixed $p\in\omega^*$, let
\[
 \lambda_{b(p)}:G\longrightarrow\Z
\]
be the algebraic coordinate homomorphism which takes the value $1$ at
the basis element $b(p)$ and $0$ at every element of
$C\setminus\{b(p)\}$. 
Set $L=\{j:\,c_{i,j}(p)=b(p)\}$. Put
\[
 a_i(p)=\lambda_{b(p)}(x_i(p))
 =\sum_{j\in L}s_{i,j}.
\tag{2.9}\label{eq:ai}
\]
Applying $\lambda_{b(p)}$ to \eqref{eq:limit-one}, we obtain
\[
 1=\sum_{i=1}^{r}q_i a_i(p)
 \qquad(p\in\omega^*).
\tag{2.10}\label{eq:diagonal}
\]
This calculation is purely algebraic and is made separately for each
$p$. No continuity of $\lambda_{b(p)}$ or of the functions $a_i$ is
asserted.

\smallskip
For the finitely many pairs $(i,j)$ define
\[
 E_{i,j}=\{p\in\omega^*:c_{i,j}(p)=b(p)\}.
\tag{2.11}\label{eq:equalizers}
\]
By Lemma~\ref{lem:ultrafilter-limit-continuity},
$E_{i,j}$ is closed in $\omega^*$.
The space $\omega^*$ is compact Hausdorff, hence Baire.
Put $J=\{(i,j):1\le i\le r,\ 1\le j\le m_i\}$.
Applying Lemma~\ref{4} to the finite family
$\{E_{i,j}:(i,j)\in J\}$, we obtain a subset $I\subseteq J$ and a
non-empty open set $U\subseteq\omega^*$ such that
\[
 U\subseteq
 \left(\bigcap_{(i,j)\in I}E_{i,j}\right)
 \cap
 \left(\bigcap_{(i,j)\in J\setminus I}
             (\omega^*\setminus E_{i,j})\right).
\]
Thus, for every $p\in U$ and every $(i,j)\in J$,
\[
 c_{i,j}(p)=b(p)\quad\Longleftrightarrow\quad(i,j)\in I.
\]
For each $i\in\{1,\ldots,r\}$, formula~\eqref{eq:ai} therefore gives
\[
 a_i(p)=\sum_{\substack{1\le j\le m_i\\(i,j)\in I}}s_{i,j}
 \qquad(p\in U),
\]
where an empty sum is understood to be zero. Both the index set and
the coefficients in this sum are independent of $p$, so each $a_i$
is constant on $U$. By \eqref{eq:diagonal}, these constant values
satisfy $\sum_{i=1}^r q_i a_i(p)=1$ on $U$, and hence at least one
of them is non-zero. Consequently, there is a fixed $i_0\le r$ such that
\[
 a_{i_0}(p)\ne0
 \qquad(p\in U).
\tag{2.12}\label{eq:nonzero}
\]

\smallskip
By \eqref{eq:x-X} and Lemma~\ref{3}, the set
\[
 \Gamma_{i_0}=\{p\in\omega^*:x_{i_0}(p)\in X\}
\]
is dense in $\omega^*$. Choose $p\in U\cap\Gamma_{i_0}$. Then
\[
 x_{i_0}(p)\in X\subseteq\langle{X}\rangle=\langle{B}\rangle.
\tag{4.13}\label{eq:inside-B}
\]
We claim that $b(p)\in B$. Suppose, towards a contradiction, that
$b(p)\notin B$. Since $B\subseteq C$, every $d\in B$ then belongs to
$C\setminus\{b(p)\}$, and hence $\lambda_{b(p)}(d)=0$ by the
definition of the coordinate homomorphism. Being a group homomorphism,
$\lambda_{b(p)}$ therefore vanishes on the entire subgroup
$\langle B\rangle$ generated by $B$. In view of \eqref{eq:inside-B},
this implies $\lambda_{b(p)}(x_{i_0}(p))=0$. On the other hand,
\eqref{eq:ai} and \eqref{eq:nonzero} give
\[
 \lambda_{b(p)}(x_{i_0}(p))=a_{i_0}(p)\ne0,
\]
a contradiction. Thus $b(p)\in B$.

Since $p$ is free and $b(p)=p\text{-}\lim b_n$, every neighbourhood of
$b(p)$ in $B$ contains infinitely many terms of the injective sequence
$(b_n)$. Thus $b(p)$ is an accumulation point of $D$ in $B$. We have
proved that every countably infinite subset of $B$ has an accumulation
point in $B$. Since $B$ is Tychonoff, this is equivalent to countable
compactness.
\end{proof}

\begin{lemma}\label{lem:beta-lift}
Let $X$ and $Y$ be pseudocompact Tychonoff spaces, and let
\[
 h:A(Y)\longrightarrow A(X)
\]
be a topological isomorphism. Put $K=\beta X$ and $L=\beta Y$. Under the
canonical identifications of $A(X)$ and $A(Y)$ with subgroups of $A(K)$
and $A(L)$, respectively, there exists a topological isomorphism
\[
 \tilde{h}:A(L)\longrightarrow A(K)
\]
such that $\tilde{h}|_{A(Y)}=h$.
\end{lemma}

\begin{proof}
By Proposition~\ref{prop:known}(4), the canonical homomorphisms
$A(X)\to A(K)$ and $A(Y)\to A(L)$ are topological monomorphisms. We make
these identifications from now on.

The subspace $h(Y)$ of $A(K)$ is pseudocompact. By
Proposition~\ref{prop:known}(3), there exists $N<\omega$ such that
\[
 h(Y)\subseteq A_N(K).
\]
Put $f=h|_Y:Y\to A_N(K)$. Since $A_N(K)$ is compact, $f$ has a
continuous extension
\[
 \tilde{f}:L\longrightarrow A_N(K)\subseteq A(K).
\]
The universal property of $A(L)$ gives a continuous homomorphism
\[
 \tilde{h}:A(L)\longrightarrow A(K)
\]
such that $\tilde{h}|_L=\tilde{f}$. In particular,
$\tilde{h}|_Y=f=h|_Y$. Since $Y$ algebraically generates $A(Y)$,
we have $\tilde{h}|_{A(Y)}=h$.

Let $g=h^{-1}:A(X)\to A(Y)$. Applying the same construction to $g$,
we obtain a continuous homomorphism
\[
 \tilde{g}:A(K)\longrightarrow A(L)
\]
such that $\tilde{g}|_{A(X)}=g$. Consequently, for every $y\in Y$,
\[
 (\tilde{g}\circ\tilde{h})(y)
 =\tilde{g}(h(y))=g(h(y))=y.
\tag{5.1}\label{eq:inverse-on-Y}
\]
The continuous maps $(\tilde{g}\circ\tilde{h})|_L:L\to A(L)$ and the
canonical inclusion of $L$ into $A(L)$ agree on the dense subspace
$Y$ of $L$. They therefore agree on all of $L$. Since $L$ algebraically
generates $A(L)$,
\[
 \tilde{g}\circ\tilde{h}=\id_{A(L)}.
\]
The symmetric argument gives
$\tilde{h}\circ\tilde{g}=\id_{A(K)}$. Thus $\tilde{h}$ is a
topological isomorphism with inverse $\tilde{g}$, and
$\tilde{h}|_{A(Y)}=h$ as required.
\end{proof}

\begin{theorem}\label{thm:A-invariant}
Let $X$ and $Y$ be Tychonoff spaces. If $A(X)$ and $A(Y)$ are
topologically isomorphic, then $X$ is countably compact if and only if
$Y$ is countably compact.
\end{theorem}

\begin{proof}
By symmetry, it suffices to prove one implication. Let
$h:A(Y)\to A(X)$ be a topological isomorphism and suppose that $X$ is
countably compact. Then $X$ is pseudocompact, and pseudocompactness is an
$A$-invariant by Proposition~\ref{prop:known}(2). Hence $Y$ is
pseudocompact.

Put $K=\beta X$ and $L=\beta Y$. By Lemma~\ref{lem:beta-lift}, $h$
extends to a topological isomorphism
\[
 \tilde{h}:A(L)\longrightarrow A(K),
 \qquad \tilde{h}|_{A(Y)}=h.
\]
Define
\[
 C=\tilde{h}(L),\qquad B=\tilde{h}(Y)=h(Y).
\]
The sets $K$ and $C$ are compact topological free Abelian bases of
$A(K)$, and $\tilde{h}|_Y=h|_Y:Y\to B$ is a homeomorphism. Moreover,
\[
 \langle{B}\rangle=\tilde{h}(A(Y))=h(A(Y))=A(X)=\langle{X}\rangle.
\tag{5.2}\label{eq:main-generated}
\]
Theorem~\ref{main} shows that $B$ is countably compact.
Therefore $Y$ is countably compact.
\end{proof}
\begin{corollary}\label{cor:M-invariant}
Let $X$ and $Y$ be Tychonoff spaces. If $F(X)$ and $F(Y)$ are
topologically isomorphic, then $X$ is countably compact if and only if
$Y$ is countably compact. In particular, the answer to
Question~\ref{q:AT} is affirmative.
\end{corollary}

\begin{proof}
By Proposition~\ref{prop:known}(1), $M$-equivalence implies
$A$-equivalence. Apply Theorem~\ref{thm:A-invariant}.
\end{proof}

\end{document}